\documentclass[11pt,hyp,]{nyjm}

\usepackage[backref=page]{hyperref}
\hypersetup{nesting=true,debug=true,naturalnames=true}
\usepackage{graphicx,amssymb,upref}
\usepackage{amsfonts,verbatim,amsmath,amsthm,latexsym,textcomp,amscd,epsfig}
\usepackage{tikz-cd}
\usepackage{graphics,textcomp}
\usepackage{color}
\usepackage{url}
\usepackage{enumitem}

\let\<\langle
\let\>\rangle
\newcommand{\doi}[1]{doi:\,\href{http://dx.doi.org/#1}{#1}}

\let\uml\"

\title{Iterations of the functor of naive $\mathbb A^1$-connected components of varieties}
\author{Nidhi Gupta}	
\address{Department of Mathematical Sciences, Indian Institute of Science Education and Research Mohali, Knowledge City, Sector-81, Mohali 140306, India.}
\email{mp18009@iisermohali.ac.in}
\thanks{The author is supported by the Prime Minister's Research Fellowship from the Ministry of Human Resource Development, Government of India.}


\newtheorem{theorem}{Theorem}[section]
\newtheorem{lemma}[theorem]{Lemma}
\newtheorem{proposition}[theorem]{Proposition}

\theoremstyle{definition}
\newtheorem{definition}[theorem]{Definition}
\newtheorem{notation}[theorem]{Notation}
\newtheorem*{ack}{Acknowledgment}

\newtheorem{remark}[theorem]{Remark}
\newtheorem{construction}[theorem]{Construction}
\newtheorem*{case}{Case}

\begin{document}

\begin{abstract}
For any sheaf of sets $\mathcal F$ on $Sm/k$, it is well known that the universal $\mathbb A^1$-invariant quotient of $\mathcal F$ is given as the colimit of sheaves $\mathcal S^n(\mathcal F)$ where $\mathcal S(F)$ is the sheaf of naive $\mathbb A^1$-connected components of $\mathcal F$. We show that these infinite iterations of naive $\mathbb A^1$-connected components in the construction of universal $\mathbb A^1$-invariant quotient for a scheme are certainly required. For every $n$, we construct an $\mathbb A^1$-connected variety $X_n$ such that $\mathcal S^n(X_n)\neq \mathcal S^{n+1}(X_n)$ and $\mathcal S^{n+2}(X_n)=*$.
\end{abstract}

\maketitle

\setcounter{tocdepth}{1}
\tableofcontents

\section{Introduction}

Let $k$ be a field and $X$ be any smooth, finite-type scheme over 
$k$. In the unstable $\mathbb A^1$-homotopy category $\mathcal H(k)$ \cite{MV}, there are two notions of $\mathbb A^1$-connectedness for $X$. The genuine notion is the sheaf of $\mathbb A^1$-connected components $\mathbb \pi_0^{\mathbb A^1}(X)$, which is given by the Nisnevich sheafication of the presheaf that associates to any smooth scheme $U$ the set of morphisms from $U$ to $X$ in $\mathcal H(k)$. The naive notion is given by the sheaf of $\mathbb A^1$-chain connected components $\mathcal S(X)$(see Definition \ref{naive definiton}). Both of these notions may not coincide even for smooth and proper schemes \cite{BHSconnected}. However, if we take infinite iterations of $\mathcal S$ and subsequently form the direct limit, the resulting sheaf $\mathcal L(X)$(also known as the universal $\mathbb A^1$-invariant quotient) will coincide with $\mathbb \pi_0^{\mathbb A^1}(X)$, provided that the latter is $\mathbb A^1$-invariant \cite[Theorem 1]{BHSconnected}. 

The $\mathbb A^1$-invariance of the sheaf of $\mathbb A^1$-connected components for a general space $\mathcal X$ in $\mathcal H(k)$ has recently been disproved \cite{Aconjecture}. Nevertheless, there are various examples of schemes where the equality of $\mathbb \pi_0^{\mathbb A^1}$ and $\mathcal L$ has been established. It is known to coincide for $\mathbb A^1$-rigid schemes, proper curves \cite{BHSconnected}, smooth projective surfaces over an algebraically closed field \cite{BSruled}, smooth projective retract rational varieties over an infinite field \cite{balwe2023naive}, etc. Moreover, $\mathcal L(X)$ provides a complete geometric description of $\pi_0^{\mathbb A^1}(X)$ for sections over finitely generated, separable field extensions of $k$ \cite[Theorem 1.1]{BRSiterations}. 

In all the above examples, $\mathcal L$ has been shown to stabilise at some finite stage. In other words, $\mathcal L$ is shown to be equal to $\mathcal S^n$ for some $n$ in all these cases.  This leads to a natural question: are these iterations really necessary? More specifically, does there exist an $n$ such that $\mathcal S^n(X)=\mathcal L(X)$ for any scheme $X$? For a general space $\mathcal X$, it has already been answered in the negative by Balwe-Rani-Sawant \cite[Theorem 1.2]{BRSiterations}. For each $n$, they have constructed a sheaf of sets for which the iterations of naive $\mathbb A^1$-connected components do not stabilise before the $n$th stage. Moreover, they have remarked on the possibility of suitably modifying their construction to produce schemes $X_n$ with the same property \cite[Remark 4.7]{BRSiterations}.

The purpose of this note is to show that the infinite iterations of naive $\mathbb A^1$-connected components in the construction of $\mathcal L$ are certainly required in the case of varieties as well and that the suggested example in op. cit. indeed works. We prove the following:

\begin{theorem}\label{main}
For each $ n\in \mathbb N$, there exists a variety $X_n$ over $\mathbb C$ of dimension $n+1$ such that $\mathcal S^n(X_n)\neq \mathcal S^{n+1}(X_n)$. 
\end{theorem}

The first example of a variety for which $\mathcal S(X)\neq \mathcal S^2(X)$ is of a singular surface $S_1$ over $\mathbb C$ \cite[Construction 4.3]{BHSconnected}. Taking $X_1=S_1$, we have inductively constructed a sequence of varieties $X_n$ having two points, $\alpha_n$ and $\beta_n$, in $X_n(\mathbb C)$ such that $\alpha_n$ and $\beta_n$ have the same images in $\mathcal S^{n+1}(X_n)(\mathbb C)$ but distinct images in $\mathcal S^n(X_n)(\mathbb C)$. We also show that these varieties $X_n$ are $\mathbb A^1$-connected and that $\pi_0^{\mathbb A^1}(X_n)=\mathcal S^{n+2}(X_n)=*$.

\begin{ack}
The author wishes to express her gratitude to her PhD supervisor, Dr. Chetan Balwe, for suggesting this problem, his constant guidance, and many helpful suggestions during the preparation of this note. She also thanks Dr. Anand Sawant for several insightful comments, which led to Theorem \ref{A1-connectedness}. Additionally, the author is grateful to the anonymous referee for their careful reading and various suggestions that improved the exposition of the paper. 
\end{ack}

\section{Preliminaries}

In this section, we recall relevant material from \cite{BHSconnected,MV} to make our exposition self-contained. We fix a base field $k$. Let $Sm/k$ denote the Grothendieck site of smooth schemes of finite type over $k$ equipped with the Nisnevich topology. 

\begin{notation}
For any smooth scheme $U$ over $k$ and $t\in k$, $s_t^U$ denotes the morphism $U\rightarrow \mathbb A^1_k\times U$ given by $u\mapsto (t,u)$. For any $H\in \mathcal{F}( \mathbb A^1_k\times U)$, define $H(t):= H\circ s_t^U$.
\end{notation}

\begin{definition}
Let $\mathcal F$ be a sheaf of sets in Nisnevich topology. For any smooth scheme $U$ in $Sm/k$ and $x_0$, $x_1$ in $\mathcal F(U)$, we say $x_0$ and $x_1$ are \textit{$\mathbb A^1$-homotopic} if there exists $h\in \mathcal F(\mathbb A^1_k\times U)$ such that $h(0)=x_0$ and $h(1)=x_1$. Moreover, $h$ is called an \textit{$\mathbb A^1$-homotopy} connecting $x_0$ and $x_1$.
\end{definition}

\begin{definition}\label{naive definiton}
The sheaf of \textit{naive $\mathbb A^1$-connected components} of $\mathcal F$, denoted by $\mathcal S(\mathcal F)$ is defined as the Nisnevich sheafication of the presheaf $\mathcal S^{pre}(\mathcal F)$, 
\[
\mathcal S^{pre}(\mathcal F)(U):= \frac{\mathcal F(U)}{\sim},
\]
where $\sim$ is the equivalence relation generated by $\mathbb A^1$-homotopy. Equivalently, $\mathcal S(\mathcal F)$ is the Nisnevich sheafication of the presheaf 
\[
U\mapsto \mathbb \pi_0\text{Sing}_*^{\mathbb A^1}(\mathcal F)(U),
\]
where $\text{Sing}^*_{\mathbb A^1}(\mathcal F)$ denotes the \textit{Morel-Voevodsky singular construction} on $\mathcal F$\cite[p.87]{MV}.  
\end{definition}

For any sheaf of sets $\mathcal F$, it is immediate from the definition of $\mathcal S$, that $\mathcal S(F)$ satisfies the following universal property.

\begin{lemma}\label{universal property of S}
Let $\mathcal F$, $\mathcal G\in Shv(Sm/k)_{Nis}$ be sheaves of sets. Suppose $\psi:\mathcal F\rightarrow G$ is a morphism such that for any $\mathbb A^1$-homotopy $h\in \mathcal F(\mathbb A^1_k\times U)$, and for any $s,t\in k$, the morphisms $(\psi\circ h)(s)$ and $(\psi\circ h)(t)$ are identical. Then $\psi$ factors through the canonical morphism $\mathcal F\rightarrow \mathcal S(\mathcal F)$.
\end{lemma}

\begin{proof}
View the morphism $\psi: \mathcal F\rightarrow G$ as a morphism of presheaves. By the definition of $\mathcal S^{pre}(\mathcal F)$, for any smooth scheme $U$, $\psi(U)$ factors through the morphism $\mathcal F(U)\rightarrow \mathcal S^{pre}(\mathcal F)(U)$:
\[
           \begin{tikzcd}
			\mathcal F(U) \arrow[r,rightarrow, "\psi(U)"] \arrow[d,rightarrow, ] & \mathcal G(U)  \\
			\mathcal S^{pre}(\mathcal F)(U) \arrow[ur,rightarrow, ]
		\end{tikzcd}
\]
Since $\mathcal F$ and $\mathcal G$ are sheaves of sets, after Nisnevich sheafication, the lemma follows.
\end{proof}

\begin{definition}
A sheaf $\mathcal F\in Shv(Sm/k)_{Nis}$ is called \textit{$\mathbb A^1$-invariant} if the maps $\mathcal F(U)\rightarrow \mathcal F(\mathbb A^1_k\times U)$, induced by the projections $\mathbb A^1_k\times U \rightarrow U$, are bijections. We say a scheme $X$ is \textit{$\mathbb A^1$-rigid} if, when viewed as a sheaf of sets, $X$ is $\mathbb A^1$-invariant.
\end{definition}

Iterating the construction of $\mathcal S$ infinitely many times yields a sequence of epimorphisms
\[
\mathcal F\rightarrow \mathcal S(F)\rightarrow \mathcal S^2(\mathcal F)\dots.
\]
After taking the direct limit, we arrive at the \textit{universal $\mathbb A^1$-invariant quotient} $\mathcal L(\mathcal F)$,
\[
\mathcal L(F):=\lim_{\rightarrow n} \mathcal S^n(\mathcal F).
\]

\begin{definition}
For any scheme $X$ over $k$, an \textit{elementary Nisnevich covering} comprises of the following two maps:
    \begin{enumerate}
        \item An open immersion $j: U\rightarrow X$.
        \item An etale map $p: V\rightarrow X$ where its restriction to $p^{-1}(X\setminus j(U))$ is an isomorphism onto $X\setminus j(U)$.
    \end{enumerate}
    
The resulting cartesian square 
\[
    \begin{tikzcd}
     U\times_X V  \arrow[d,rightarrow]\arrow[r,rightarrow] & 
     V \arrow[d, rightarrow, "p"]\\
     U  \arrow[r,rightarrow,"j" ] & 
     X
\end{tikzcd}
\]
is called an \textit{elementary distinguished square}.
\end{definition}

One of the significant aspects of using an elementary Nisnevich covering is illustrated by the following result in \cite[\S 3, Lemma 1.6 ]{MV}.

\begin{lemma}\label{pushout}
An elementary distinguished square is a cocartesian square in the category $Shv(Sm/k)_{Nis}$.
\end{lemma}

The lemma mentioned in \cite{MV} is originally stated for smooth schemes; however, the same proof holds for general schemes without any modifications. We will recall some results from \cite{ASnaive,BRSiterations,MF} that will be used to prove the $\mathbb A^1$-connectedness of $X_n$. The following lemma is a standard result from \cite[Lemma 6.1.3]{MF}.

\begin{lemma}\label{stable}
A sheaf of sets $\mathcal F$ on $Sm/k$ is $\mathbb A^1$-connected if $\mathbb \pi_0^{\mathbb A^1}(\mathcal F)(K)=*$ for any finitely generated  separable extension $K$ of $k$.
\end{lemma}

The following theorem from \cite[Theorem 2.2]{BRSiterations} provides an explicit formula for computing $\pi_0^{\mathbb{A}^1}(\mathcal{F})(K)$ for any field $K/k$.

\begin{theorem}\label{formula}
Let $\mathcal F$ be a sheaf of sets. For any finitely generated field extension $K/k$, the natural map $\pi_0^{\mathbb A^1}(\mathcal F)(K) \rightarrow \mathcal L(\mathcal F)(K)$ is a bijection.
\end{theorem}

The analogue of Lemma \ref{stable} for $\mathcal{S}$ is given by the following result from \cite[Theorem 3.2]{ASnaive}.

\begin{theorem}\label{S^2}
Suppose $\mathcal F$ is a sheaf on sets such that $\mathcal S(\mathcal F)(K)=*$ for any finitely generated separable $K/k$. Then, $\mathcal S^2(\mathcal F)=*$. 
\end{theorem}

\begin{notation}\label{sq}
For any sheaf of sets $\mathcal F$ and $x$ in $\mathcal F(U)$,  we will use $[x]_j$ to denote the image of $x$ in $\mathcal S^j(\mathcal F)(U)$. 
\end{notation}

\begin{notation} 
From now on, all schemes are defined over $\mathbb C$.  
For any two schemes $X$ and $Y$, and for any $x\in X(\mathbb C)$, we will denote the morphism $Y\rightarrow \text{Spec }{\mathbb C}\xrightarrow{x} X$ by $Y\xrightarrow{x} X$.
\end{notation}

\section{Proof of Theorem \ref{main}}

We divide the proof into three parts. First, we construct the required sequence $X_n$ of varieties and fix two $\mathbb C$-valued points $\alpha_n$ and $\beta_n$ in $X_n$. Second, we provide geometric arguments to show that the images of $\alpha_n$ and $\beta_n$ cannot be equal in $\mathcal S^i(X_n)(\mathbb C)$ for $i\leq n$. Finally, we  use appropriate elementary Nisnevich covers of $X_n$ to construct maps from $X_n$ to $\mathcal S(X_{n+1})$ that map $\alpha_n$ and $ \beta_n$ to $[\beta_{n+1}]_1$ and $[\alpha_{n+1}]_1$(see Notation \ref{sq}), respectively. This allows us to construct an $\mathbb A^1$-homotopy connecting $[\alpha_n]_n$ and $[\beta_n]_n$ in $\mathcal S^n(X_n)$, thereby ensuring that the images of $\alpha_n$ and $\beta_n$ are equal in $\mathcal S^{n+1}(X_n)$.

\subsection{Construction of the varieties $X_n$}
 
For $n\geq 0$, the variety $X_n$ is quasi affine and has dimension $n+1$. In proving the theorem, it will be useful to have the explicit equations defining $X_n$. Therefore, we will  construct the affine varieties $Y_n$ in $\mathbb A^{2n+1}_{\mathbb C}$, such that $X_n$ is an open subvariety of $Y_n$. We begin by constructing the varieties $X_n$. Set $X_0=Y_0:=\mathbb A^1_{\mathbb C}$. 

\begin{construction}\label{construct X_1}
We now recall the construction of surface $S_1$ from \cite[Construction 4.3]{BHSconnected} which will serve as our $X_1$.

\begin{enumerate}
\item Let $\lambda_i\in \mathbb C\setminus 0$ for $i=1,2,3$, and let $f(x_1)=(x_1-\lambda_1)(x_1-\lambda_2)(x_1-\lambda_3)$ with $\lambda =\sqrt{-\lambda_1\lambda_2\lambda_3}$. Define $E$ as the following planar curve,
    \[
    E:=\text{Spec }\mathbb C[x_1,y_1]/\langle y_1^2-f(x_1)\rangle.
    \]
Let $\pi:E\rightarrow \mathbb A^1$  be the projection onto $x_1$-axis. Thus, $\pi^{-1}(0)=\{(0,\pm\lambda)\}$.

\item Define $Y_1$ and $X_1$ as the following surfaces in $\mathbb A^3_{\mathbb C}$,
\[Y_1:=\text{Spec } \mathbb C[x_0,x_1,y_1]/\langle y_1^2-x_0^2f(x_1)\rangle\text{ , }  X_1:= Y_1 \setminus \{(0,0,0)\}.\] Let $i_1$ denote the inclusion of $X_1$ into $Y_1$.

\item Let $\bar\phi_1$ and $\bar\psi_1: Y_1\rightarrow X_0$  be the 
projection onto the $x_0$-axis and $x_1$-axis, respectively. Define $\phi_1$ and $\psi_1$ as the restrictions of $\bar\phi_1$ and $\bar\psi_1$ to $X_1$. The surface $X_1$ can be viewed as a family of curves parametrized by $\mathbb A^1_{\mathbb C}$ via $\phi_1$, where the fiber over $0$ is $\mathbb G_m$ and the fiber over any nonzero point is $E$. 

\item Let $\alpha_1=(1,0,\lambda)$ and $\beta_1=(0,1,0)$. Then, $\alpha_1$ is contained in the copy of $E$ in $X_1$ corresponding to $x_0=1$, while $\beta_1$ is contained in the copy of $\mathbb G_m$ in $X_1$ corresponding to $x_0=0$. $\mathbb A^1$-rigidity of $E$ and $\mathbb G_m$ will be used to show that $\alpha_1$ and $\beta_1$ cannot be connected by a chain of $\mathbb A^1_{\mathbb C}$ in $X_1$.

\item Let $\mathbb A^1_{\mathbb C}\times E$ denote the surface $\text{Spec}\mathbb C[x_0,x_1,y_1]/\langle y_1^2-f(x_1)\rangle$, and let $\rho_1:\mathbb A^1_{\mathbb C}\times E\rightarrow Y_1 $ be the morphism given by $(x_0,x_1,y_1)\mapsto(x_0,x_1,x_0y_1)$. Then, $\rho_1$ is an isomorphism outside $\bar\phi_1^{-1}(\{0\})$. $\rho_1$ will be used to construct an $\mathbb A^1$-homotopy in $\mathcal S(X_1)$ connecting $[\alpha_1]_1$ and $[\beta_1]_1$.

\end{enumerate}
\end{construction}

Next, we inductively define the varieties $X_n$ and the morphisms $\phi_n,\psi_n: X_n\rightarrow X_{n-1}$.

\begin{construction}\label{construct X_n}
      For $n\geq 2$, assuming that $X_{n-1}$, $ \phi_{n-1}$ and $\psi_{n-1}$ are defined, we define $X_n$ as the pullback of the following diagram.
		\[
			\begin{tikzcd}
				X_{n} \arrow[r,rightarrow, "\psi_n"] \arrow[d,swap,rightarrow, "\phi_n"] & X_{n-1}\arrow[d,rightarrow, "\phi_{n-1}"] \\
				X_{n-1} \arrow[r,swap,rightarrow, "\psi_{n-1}"] & X_{n-2}
			\end{tikzcd}
		\] 
\end{construction}

We now define $\alpha_n$ and $\beta_n$. These are the $\mathbb C$-valued points of $X_n$ whose images are not equal in $\mathcal S^n(X_n)(\mathbb C)$ but will be equal in $\mathcal S^{n+1}(X_{n})(\mathbb C)$. We start by defining $\alpha_2$ and $\beta_2$. Since $\psi_1(\alpha_1)=\phi_1(\beta_1)=0$, the pair $(\alpha_1,\beta_1)$ induces the morphism 
\[
\text{Spec }\mathbb C\xrightarrow{(\alpha_1,\beta_1)}X_1\times_{\psi_1,X_0,\phi_1}X_1.
\] 
Similarly, since $\psi_1(\beta_1)=\phi_1(\alpha_1)=1$, the pair $(\beta_1,\alpha_1)$ will induce the morphism
\[
\text{Spec }\mathbb C\xrightarrow{(\beta_1,\alpha_1)}X_1\times_{\psi_1,X_0,\phi_1}X_1.
\] 
Define $\alpha_2,\beta_2:\text{Spec }\mathbb C\rightarrow X_2$ by the following morphisms,
\[
\alpha_2:=(\alpha_1,\beta_1)\quad \text{and} \quad\beta_2:=(\beta_1,\alpha_1).
\]
To define $\alpha_n$ and $\beta_n$ for $n\geq 3$, an alternative definition of $X_n$ will be more convenient.

\begin{remark}\label{fold}
$X_n$ can be realised as the $n$-fold fiber product of $X_1$ over $\mathbb A^1_\mathbb C$ as follows:
For $i\geq 2$, we see immediately that $X_i=X_{1}\times_{\mathbb A^1_{\mathbb C}}X_{i-1}$ via the following pullback square.
            \[
			\begin{tikzcd}
				X_{i}  \arrow[d,swap,rightarrow, "\phi_{2}\circ\dots\phi_i"] \arrow[r,rightarrow, "\psi_{i}"] & X_{i-1} \arrow[d, rightarrow, "\phi_{1}\circ\dots\phi_{i-1}"]\\
				X_{1}  \arrow[r,,swap,rightarrow,"\psi_{1}" ] & \mathbb A^1_{\mathbb C}
			\end{tikzcd}
		\]
Applying this definition of $X_i$ repeatedly, we will find that $X_n$ can also be expressed as the $n$-fold fiber product, specifically, $X_n=X_1\times_{\psi_1,\mathbb A^1_{\mathbb C},\phi_1} X_1\dots \times_{\psi_1,\mathbb A^1_{\mathbb C},\phi_1} X_{1}$. 
\end{remark}

For all $n\geq 3$, $\alpha_n$ and $\beta_n$ are defined using the above $n$-fold fiber product description of $X_n$. The morphisms are given by
\[
\alpha_n : \text{Spec }\mathbb C\xrightarrow{(a_1,\dots,a_n)} X_n \quad\text{and}\quad \beta_n :\text{Spec }\mathbb C\xrightarrow{(b_1,\dots,b_n)}  X_n,
\]
where $ a_i,b_i\in\{\alpha_1,\beta_1\}$  such that
$(a_1,\dots,a_n)$ and $(b_1,\dots,b_n)$ form alternating sequences of $\alpha_1$ and $\beta_1$ with $a_1=\alpha_1$ and $b_1=\beta_1$. More precisely, 
\[
\alpha_n:=(\alpha_1,\beta_1,\alpha_1,\dots,a_n)\quad\text{and}\quad\beta_n:=(\beta_1,\alpha_1,\beta_1,\dots,b_n).
\]
Since $\psi_1(\beta_1)=\phi_1(\alpha_1)=1$ and $\psi_1(\alpha_1)=\phi_1(\beta_1)=0$, the above definitions are well defined. Similar to the definition of $X_n$, we define $Y_n$ inductively.

\begin{construction}\label{construct Y_n} 
For $n\geq 2$, assuming $Y_{n-1}$, $ \bar{\phi}_{n-1}$ and $\bar \psi_{n-1}$ are defined, we define $Y_n$ using the following pullback square.
		\[
			\begin{tikzcd}
				Y_{n} \arrow[r,rightarrow, "\bar \psi_n"] \arrow[d,swap,rightarrow, "\bar \phi_n"] & Y_{n-1}\arrow[d,rightarrow, "\bar \phi_{n-1}"] \\
				Y_{n-1} \arrow[r,swap,rightarrow, "\bar \psi_{n-1}"] & Y_{n-2}
			\end{tikzcd}
		\]
\end{construction}

The following simple lemma provides a geometric description of $X_n$ and $Y_n$.

\begin{lemma}\label{geometric}
Let $n\geq 1$.
\begin{enumerate}

\item $X_n$ is an open subscheme of $Y_n$. Moreover, $\phi_n$ and $\psi_n$ are  the restrictions of $\bar\phi_n$ and $\bar\psi_n$ to $X_n$.
        
\item $Y_n=\text{Spec }\mathbb C[x_0,x_1,y_1\dots,x_n,y_n]/\langle\{y_i^2-x_{i-1}^2f(x_i)\}^n_{i=1}\rangle$. The morphism $\bar\phi_n$ is given by 
\[
(x_0,x_1,y_1,\dots x_n, y_{n})\mapsto (x_0,x_1,y_1\dots, x_{n-1},y_{n-1}),
\]
and the morphism $\bar\psi_n$ is given by 
\[
(x_0,x_1,y_1,\dots x_n, y_{n})\mapsto (x_1,x_2,y_2\dots, x_{n},y_{n}).
\]

\end{enumerate}
\end{lemma}

\begin{proof}
 For (1), we first claim that similar to $X_n$, $Y_n$ can be obtained as the $n$-fold fiber product of $Y_1$ over $\mathbb A^1_\mathbb C$. Indeed, by replacing the roles of $\phi_i$ and $\psi_i$ with $\bar\phi_i$ and $\bar\psi_i$ in Remark \ref{fold}, we find that 
\[
Y_n=Y_1\times_{\bar\psi_1,\mathbb A^1_{\mathbb C},\bar\phi_1} Y_1\dots \times_{\bar\psi_1,\mathbb A^1_{\mathbb C},\bar\phi_1} Y_{1}.
 \]
Now, since $i_1:X_1\rightarrow Y_1$ is an open immersion, and since $\phi_1$ and $\psi_1$ are simply the restrictions of $\bar\phi_1, \bar\psi_1$ to $X_1$ respectively,  the following morphism
\[
X_1\times_{\psi_1,\mathbb A^1_{\mathbb C},\phi_1} X_1\dots \times_{\psi_1,\mathbb A^1_{\mathbb C},\phi_1} X_{1}\xrightarrow{i_1\times\dots \times i_1}Y_1\times_{\bar\psi_1,\mathbb A^1_{\mathbb C},\bar\phi_1} Y_1\dots \times_{\bar\psi_1,\mathbb A^1_{\mathbb C},\bar\phi_1} Y_{1}
\]
must be an open immersion. Now, $\phi_n,\psi_n:X_{n-1}\times_{\psi_{n-1},X_{n-2},\phi_{n-1}}X_{n-1}\rightarrow X_{n-1}$ are the first and second projections onto $X_{n-1}$, respectively. Therefore, in the $n$-fold fiber product description,  $\phi_n$ and $\psi_n$ will project $X_n$ onto the first and last $n-1$ factors of $X_n$, respectively. Similarly, $\bar \phi_n$ and $\bar \psi_n$ will project $Y_n$ onto the first and last $n-1$ factors of $Y_n$, respectively. It follows that the morphisms $\phi_n$, $\psi_n$ are precisely the restrictions of $\bar{\phi_n}, \bar{\psi_n}$ to $X_n$.

For (2), because  $Y_n=Y_1\times_{\bar\psi_1,\mathbb A^1_{\mathbb C},\bar\phi_1} Y_1\dots \times_{\bar\psi_1,\mathbb A^1_{\mathbb C},\bar\phi_1} Y_{1}$, the coordinate ring $A_n$ of $Y_n$ is given by
\[
A_n =\frac{\mathbb C[x_0,x_1,y_1]}{\langle y_1^2-x_{0}^2f(x_1)\rangle}\otimes_{\bar\psi_1^*,\mathbb C[x],\bar\phi_1^*}\frac{\mathbb C[x_0,x_1,y_1]}{\langle y_1^2-x_{0}^2f(x_1)\rangle}\dots\otimes_{\bar\psi_1^*,\mathbb C[x],\bar\phi_1^*}\frac{\mathbb C[x_0,x_1,y_1]}{\langle y_1^2-x_{0}^2f(x_1)\rangle}.
\]
Since $\bar\phi_1$ and  $\bar\psi_1$ project onto the $x_0$ and $x_1$ axis, respectively, this tensor product equals
\[
A_n=\text{Spec }\mathbb C[x_0,x_1,y_1,x_2,y_2,\dots,x_n,y_n]/\langle\{y_i^2-x_{i-1}^2f(x_i)\}^n_{i=1}\rangle.
\]
Moreover, since $\bar \phi_n$ and $\bar \psi_n$ will project $Y_n$ onto the first and last $n-1$ factors of $Y_n$, respectively, $\bar \phi_n$ is the projection onto the coordinates $(x_0,x_1,y_1\dots, x_{n-1},y_{n-1})$,  and $\bar \psi_n$ is the projection onto the coordinates $(x_1,x_2,y_2\dots, x_{n},y_{n})$. 
\end{proof}

\subsection{Geometric properties of $X_n$}

In this subsection, we apply the universal property of $\mathcal S$ from Lemma \ref{universal property of S} to the morphism $\psi_n$ in order to construct maps $\mathcal S(X_n)\rightarrow X_{n-1}$. Then, using an inductive argument, we show that the images of $\alpha_n$ and $\beta_n$ cannot be equal in $\mathcal S^i(X_n)(\mathbb C)$ for $i\leq n$.

\begin{lemma}\label{fiber}
 Let $n\geq 1$. Then the fibers of closed points under $\phi_n$ are $\mathbb A^1$-rigid.
\end{lemma}

\begin{proof}
From Lemma \ref{geometric}, it follows that $\phi_n$ is the restriction of the affine map $Y_n \rightarrow Y_{n-1}$ corresponding to the natural ring homomorphism of coordinate rings:
\[
\frac{\mathbb C[x_0,x_1,y_1,\dots,x_{n-1},y_{n-1}]}{\langle\{y_i^2-x_{i-1}^2f(x_i)\}^{n-1}_{i=1}\rangle}\xrightarrow{\bar \phi_n^*} \frac{\mathbb C[x_0,x_1,y_1,\dots,x_{n-1},y_{n-1}][x_{n},y_{n}]}{\langle\{y_i^2-x_{i-1}^2f(x_i)\}^{n-1}_{i=1}\rangle+\langle y_{n}^2-x_{n-1}^2f(x_n)\rangle}. 
\]
Now, let $Q$ be a closed point of $X_{n-1}$. Thus, $Q=(a_0,a_1,b_1,\dots,a_{n-1},b_{n-1})$ where $a_i,b_i\in \mathbb C$. Hence, $\phi_n^{-1}(Q)$ is isomorphic to $\text{Spec }\frac{\mathbb C[x_n,y_n]}{\langle y_n^2-a_{n-1}^2 f(x_n) \rangle} \setminus (0,0,0)$. If $a_{n-1}\neq 0$, then $\phi_n^{-1}(Q)$ is isomorphic to $E$, and otherwise, it is isomorphic to $\text{Spec }\frac{\mathbb C[x_n,y_n]}{\langle y_n^2 \rangle} \backslash (0,0,0)$. Since both of these varieties are $\mathbb A^1$-rigid, this completes the proof.
\end{proof}

\begin{lemma}\label{constant}
Let $n\geq 1$ and let $\gamma$ be any morphism $\mathbb A^1_{\mathbb C}\rightarrow X_n$. Then $\psi_n \circ \gamma$ is a constant morphism.   
\end{lemma}

\begin{proof}
We prove this by induction on $n$. Let's verify the base case for $n=1$. Let $\gamma$ be a morphism $\mathbb A^1_{\mathbb C}\rightarrow X_1$. Recall that, $\rho_1: \mathbb A^1_C \times E \rightarrow Y_1$(see Construction \ref{construct X_1},(5)) defined by $(x_0,x_1,y_1)\mapsto (x_0,x_1,x_0y_1)$ is an isomorphism outside the fiber $\bar\phi_1^{-1}(0)$. Therefore, it induces a rational map $X_1\dashrightarrow \mathbb A^1_{\mathbb C}\times E$. Since $\psi_1$  is the projection onto the $x_1$- axis, $\psi_1$ is the same as the morphism induced by the rational map $X_1\dashrightarrow \mathbb A^1_{\mathbb C}\times E\rightarrow E\xrightarrow{\pi} \mathbb A^1_{\mathbb C}$. Now, either the image of $\psi_1\circ\gamma$ lies completely in the fiber $\phi_1^{-1}(0)$, or $\psi_1\circ\gamma$ factors through the rational map $\mathbb A^1_{\mathbb C}\dashrightarrow E$, which can be completed to a morphism $\mathbb A^1_{\mathbb C}\rightarrow \bar E$, where $\bar E$ is the projective closure of $E$ . Since both $\phi_1^{-1}(0)$ and  $\bar E$ are $\mathbb A^1$-rigid, this completes the argument for the case $n=1$.
    
Assuming the lemma holds for $n-1$, we will prove it for $n$. Let $\gamma:\mathbb A^1_{\mathbb C}\rightarrow X_n$ be fixed. Recall that, $X_n$ is defined by the following Cartesian square:
    \[
			\begin{tikzcd}
				X_{n} \arrow[r,rightarrow, "\psi_n"] \arrow[d,swap,rightarrow, "\phi_n"] & X_{n-1}\arrow[d,rightarrow, "\phi_{n-1}"] \\
				X_{n-1} \arrow[r,swap,rightarrow, "\psi_{n-1}"] & X_{n-2}
			\end{tikzcd}
    \]
    
Define $\gamma_1:= \phi_n\circ\gamma$ and $\gamma_2:= \psi_n\circ\gamma$. We aim to show that $\gamma_2$ is a constant morphism. By the induction hypothesis, $\psi_{n-1}\circ\gamma_1$ is constant. From the commutativity of the square above, it follows that $\phi_{n-1}\circ\gamma_2$ is constant. This implies that the image of $\gamma_2$ lies in a fiber of $\phi_{n-1}$, which is $\mathbb A^1$-rigid by Lemma \ref{fiber}. Therefore, it follows that $\gamma_2$ must be a constant morphism.
\end{proof}

\begin{lemma}\label{factor}
The morphism $\psi_n:X_n\rightarrow X_{n-1}$ in $Shv(Sm/k)_{Nis}$ factors through the epimorphism $X_n\rightarrow \mathcal S(X_n)$. 
\end{lemma}

\begin{proof}
By Lemma \ref{universal property of S}, it suffices to show that for any smooth scheme $U$ and any $\mathbb A^1$-homotopy $F\in X_n(\mathbb A^1\times U)$, $\psi_n\circ F$ is a constant $\mathbb A^1$-homotopy.
Let  $G:=\psi_n\circ F$ and let $s,t\in \mathbb C$. We need to show that the morphisms $G(t),G(s):U\rightarrow X_{n-1}$ are identical. Since $X$ is separated, the set $S:=\{x\in U|G(t)(x)=G(s)(x)\}$  forms a closed subscheme of $U$. From Lemma \ref{constant}, we know that $U(\mathbb C)\subset S$, which further implies that $U=S$. Hence, $G(s)=G(t)$ for any $s,t\in \mathbb C$, and the result follows. 
\end{proof}

\begin{theorem}\label{neq}
$[\alpha_n]_n \neq [\beta_n]_n$ for all $n$.
\end{theorem}

\begin{proof}
We prove the theorem by induction on $n$. 
For $n=1$, we need to show that $[\alpha_1]_1$ and $[\beta_1]_1$ cannot be connected by an $\mathbb A^1$-chain homotopy. To establish this, it suffices to show that any morphism $\gamma :\mathbb A^1_{\mathbb C}\rightarrow X_1$ containing $\alpha_1=(1,0,\lambda)$ in its image is a constant morphism. Recall that, the morphism $\rho_1: \mathbb A^1_{\mathbb C}\times_{\mathbb C}E\rightarrow Y_1$ defined by $(x_0,x_1,y_1)\rightarrow (x_0,x_1,x_0y_1)$ is an isomorphism outside $\bar\phi_1^{-1}(\{0\})$. Since $\alpha_1\notin \bar \phi_1^{-1}(\{0\})$, $\rho_1^{-1}$ induces a rational map $\gamma':\mathbb A^1_{\mathbb C}\dashrightarrow \mathbb A^1_{\mathbb C}\times_{\mathbb C}E\rightarrow E $. This rational map can be completed to a morphism $\mathbb A^1_{\mathbb C}\rightarrow \bar E$. Consequently, $\gamma'$ is constant, implying that the image of $\gamma$ is contained in affine line corresponding to $\rho_1(\mathbb A^1_{\mathbb C}\times(0,\lambda))$. Since $\rho((0,0,\lambda))=(0,0,0)$ is not in $X_1$, the image of $\gamma$ must be contained in affine line excluding origin, which is $\mathbb A^1$-rigid. Thus, $\gamma$ must be a constant morphism.

Assuming the theorem holds for $n-1$, we will prove it for $n$.
On the contrary, assume that $[\alpha_n]_{n}=[\beta_n]_{n}$ in $\mathcal S^{n}(X_n)(\mathbb C)$. 
Since the morphism $\psi_n:X_n\rightarrow X_{n-1}$ factors through the morphism $X_n\rightarrow \mathcal S(X_n)$ by Lemma \ref{factor}, we obtain the following commutative diagram:
\[
		\begin{tikzcd}[column sep=large]
			\mathcal S^{n-1}(X_{n})(\mathbb C) \arrow[r,rightarrow, "\mathcal S^{n-1}(\psi_n)"] \arrow[d,rightarrow, ] & \mathcal S^{n-1}(X_{n-1})(\mathbb C)  \\
			\mathcal S^{n}(X_{n}) (\mathbb C)\arrow[ur,rightarrow]
		\end{tikzcd}
\]
Since $\psi_n(\alpha_n)=\beta_{n-1}$ and $\psi_n(\beta_n)=\alpha_{n-1}$, and we have assumed that $[\alpha_n]_{n}=[\beta_n]_{n}$, it follows from the commutativity of the above diagram that $[\alpha_{n-1}]_{n-1}=[\beta_{n-1}]_{n-1}$. This conclusion contradicts the induction hypothesis. Therefore, the theorem holds.
\end{proof}

\subsection{$\mathbb A^1$-homotopies in $\mathcal S^n(X_n)$}

An explicit $\mathbb A^1$-homotopy between $[\alpha_1]_1$ and $[\beta_1]_1$ in $\mathcal S(X_1)$ has been constructed in \cite[Construction 4.3]{BHSconnected}. This will be the key input in the following construction of $\mathbb A^1$-homotopies in $\mathcal S^n(X_n)$.

\begin{theorem}\label{homotopy}
$[\alpha_n]_n$ and $[\beta_n]_n$ are $\mathbb A^1$-homotopic in $\mathcal S^n(X_n)$. 
\end{theorem}

\begin{proof}
For every $n\geq 1$, we will construct an $\mathbb A^1$-homotopy $\mathbb A^1_{\mathbb C}\rightarrow \mathcal S^n(X_n)$ such that $[\alpha_n]_n$ and $[\beta_n]_n$ are contained in its image. We begin by constructing an elementary Nisnevich cover of $X_n$ for all $n\geq 0$. Let $V=V_1\sqcup V_2$, where
\[V_1= E\setminus \{{( \lambda_i,0)}^{3}_{i=1}(0,- \lambda)\} \quad\text{and}\quad V_2=\mathbb A^1_{\mathbb C}\setminus \{0\}.\] 
Define $p_1:=\pi|_{V_1}$ and $p_2$ to be the inclusion $V_2\rightarrow \mathbb A^1_{\mathbb C}$. The morphism $p_1$ is etale and $p_1^{-1}(0)=(0,\lambda)$, thus the map $p_1\sqcup p_2$ forms a Nisnevich cover of $\mathbb A^1_\mathbb C$. 
Now, for $n\geq 1$, consider $X_n$ as schemes over $\mathbb A^1_{\mathbb C}$ through $\Phi_n:=\phi_{1}\circ\dots\phi_n$. We then obtain the following elementary distinguished square,
   
       	\[
			\begin{tikzcd}
				W\times_{\mathbb A^1_{\mathbb C}}X_n\arrow[r,rightarrow,"pr_2"] \arrow[d,swap,rightarrow,"pr_1"] & V_2\times_{\mathbb A^1_{\mathbb C}}X_n\arrow[d,rightarrow, "p_2\times id"] \\
				V_1\times_{\mathbb A^1_{\mathbb C}}X_n \arrow[r,swap,rightarrow, "p_1\times id"] & \mathbb A^1_{\mathbb C}\times_{\mathbb A^1_{\mathbb C}}X_n
			\end{tikzcd}
		\]
where $W=V_1\times_{\mathbb A^1_{\mathbb C}}V_2$.
  
We now construct maps from $X_n$ to $\mathcal S(X_{n+1})$ that send $\alpha_n$ and $\beta_n$ to $[\beta_{n+1}]_1$ and $[\alpha_{n+1}]_1$, respectively.
Since the above square is cocartesian in $Shv(Sm/k)_{Nis}$ by Lemma \ref{pushout}, it suffices to construct morphisms $h_i^n:V_i\times_{\mathbb A^1_{\mathbb C}}X_n \rightarrow X_{n+1}$ for $i=1$ and $2$, such that the following two compositions are identical: 
\[W\times_{\mathbb A^1_{\mathbb C}}X_n\xrightarrow{pr_i\circ h_i^n} X_{n+1}\rightarrow \mathcal S(X_{n+1})\quad\text{for }i=1,2.\] By Remark \ref{fold}, $X_{n+1}=X_1\times_{\psi_1,\mathbb A^1_{\mathbb C},\Phi_n}X_{n}$. Thus, we will define $h_i^n:V_i\times_{\mathbb A^1_{\mathbb C}}X_n\rightarrow X_{n+1}$ as $h_i\times id$, where 

\begin{itemize}
    \item $h_1:V_1\rightarrow X_1\text{ ; } (x_1,y_1)\mapsto (1,x_1,y_1)$,
    \item $h_2:V_2\rightarrow X_1\text{ ; } (x_1)\mapsto (0,x_1,0)$.
\end{itemize}
The maps $h_i^n$ are well defined because $\psi_1\circ h_i=p_i$. 

Now, we will define $\mathcal H^n: \mathbb A^1_{\mathbb C}\times_{\mathbb C}(W\times_{\mathbb A^1_{\mathbb C}}X_n)\rightarrow X_{n+1}$ such that $\mathcal H^n(0)=pr_2\circ h_2^n$ and $\mathcal H^n(1)=pr_1\circ h_1^n$. Similar to $h_i^n$, $\mathcal H^n$ is a product of two morphisms,  $\mathcal H\times id:(\mathbb A^1_{\mathbb C}\times_{\mathbb C}W)\times_{\mathbb A^1_{\mathbb C}}X_n\rightarrow X_{n+1}$, where $\mathcal H$ is the restriction of $\rho_1$:
\[\mathcal H: \mathbb A^1_{\mathbb C}\times_{\mathbb C}W\rightarrow X_1\text \quad   ;   \quad(x_0,x_1,y_1)\mapsto (x_0,x_1,x_0y_1). \]
Clearly, $\mathcal H^n(0)=pr_2\circ h_2^n$ and $\mathcal H^n(1)=pr_1\circ h_1^n$. Therefore, the morphisms $pr_2\circ h_2^n$ and $pr_1\circ h_1^n$ become identical in $\mathcal S(X_{n+1})$. Thus, $h_1^n$ and $h_2^n$ can be glued together to obtain the maps $\mathcal F_n:X_n\rightarrow \mathcal S(X_{n+1})$ for all $n\geq 0$. 

Finally, for $m\geq 1$, define the required $\mathbb A^1$-homotopy in $\mathcal S^m(X_m)$ as the following composition:
\[\mathbb A^1_{\mathbb C}\xrightarrow{\mathcal F_0} \mathcal S(X_1)\xrightarrow{\mathcal S(\mathcal F_1)} \mathcal S^2(X_2)\xrightarrow{\mathcal S^2( \mathcal F_2)} \mathcal S^3(X_3)\rightarrow\dots\rightarrow \mathcal S^m(X_m).\]
To ensure that the above $\mathbb A^1$-homotopy indeed connects $[\alpha]_m$ and $[\beta_m]_m$ in $\mathcal S^m(X_m)$, what remains is to show that:
\begin{enumerate}
    \item $\mathcal F_0(0)=[\alpha_1]_1$ and $\mathcal F_0(1)=[\beta_1]_1$,
    \item For all $n\geq1$, $\mathcal F_n(\beta_n)=[\alpha_{n+1}]_{1}$ and $\mathcal F_n(\alpha_n)=[\beta_{n+1}]_{1}$.
\end{enumerate}
Since $h_1(0,\lambda)=\alpha_1$ and $h_2(1)=\beta_1$, we obtain (1).
For (2), consider the following commutative diagram:
\[
\begin{tikzcd}
 &(V_1\times_{\mathbb A^1_{\mathbb C}}X_n) \sqcup (V_2\times_{\mathbb A^1_{\mathbb C}}X_n)\arrow[r,rightarrow,"h_{1}^n\sqcup h_2^{n}"]\arrow[d,rightarrow,""] & X_{n+1}\arrow[d,rightarrow] \\
 \text{Spec }\mathbb C\sqcup \text{Spec }\mathbb C\arrow[ur,rightarrow,"((0\text{,}\lambda)\text{,}\beta_n)\sqcup(1\text{,}\alpha_n)"]
\arrow[r,rightarrow,swap,"\beta_n\sqcup\alpha_n"] & X_n \arrow[r,swap, rightarrow,"\mathcal F_n"] & \mathcal S(X_{n+1})
\end{tikzcd}
\]
Since $h_1^n((0,\lambda),\beta_n)=(\alpha_1,\beta_n)=\alpha_{n+1}$ and
$h_2^n((1),\alpha_n)=(\beta_1,\alpha_n)=\beta_{n+1}$, we are done.
\end{proof}

\begin{proof}[Proof of Theorem \ref{main}]
By Theorem \ref{neq} and Theorem \ref{homotopy}, we have 
\[ [\alpha_n]_n \neq [\beta_n]_n\quad\text{and}\quad[\alpha_n]_{n+1} =[\beta_n]_{n+1}.\]
Hence, $\mathcal S^n(X_n)(\mathbb C)\neq \mathcal S^{n+1}(X_n)(\mathbb C) $.
\end{proof}

\begin{remark}
For any scheme $X/k$, the field value sections of the sheaf of $\mathbb A^1$-connected components of $X$ can be computed by the following formula \cite{BRSiterations}:  
\[\pi_0^{\mathbb A^1}(X)(F)=\mathcal L(X)(F):= \lim_{\rightarrow n} \mathcal S^n(\mathcal X)(F)\quad\text{for any }F/k.\] 
If $X$ is proper, then $\mathcal S(X)(F)=\mathcal S^2(X)(F)$. However, for non-proper $X$,
the proof of Theorem \ref{main} shows that the infinite iterations of $\mathcal S$ in the above formula are essential. 
\end{remark}

In \cite[Question 2.16]{CRA2}, it is asked whether $\mathcal S(X)(F)=\mathcal S^2(X)(F)$ for non-proper smooth schemes over $k$ when $k=\bar{k}$. This question was already answered in the negative in \cite[Construction 4.5]{BHSconnected}, where  $X_1$ was embedded in a smooth variety $T$ using  \cite[Lemma 4.4]{BHSconnected}, such that $\mathcal S(T)(\mathbb C)\neq \mathcal S^2(T)(\mathbb C)$. We will use a slight generalisation of the same lemma(see Lemma \ref{rigid embedding}) to embed $X_n$ in the smooth varieties $Z_n$ such that $\mathcal S^n(Z_n)(\mathbb C)\neq\mathcal S^{n+1}(Z_n)(\mathbb C)$. This shows that the infinite iterations of $\mathcal S$ are required even for field value points of non-proper smooth schemes over an algebraically closed field.

The following lemma is a reformulation of \cite[Lemma 2.12]{BSruled} and will be used in the construction of $Z_n$.

\begin{lemma}\label{homotopies respect fibers}
    Let $\phi:\mathcal F\rightarrow G$ be a morphism of sheaves of sets on $Sm/k$. Assume that $\mathcal G$ is $\mathbb A^1$-invariant. Then, for any $n$ and any $\mathbb A^1$-homotopy $h:\mathbb A^1\times U\rightarrow \mathcal S^n(\mathcal F)$, there exists $\gamma:U\rightarrow G$ such that the given $\mathbb A^1$-homotopy factors through $\mathcal S^n(\mathcal F\times_{\mathcal G,\gamma}U)\rightarrow \mathcal S^n(\mathcal F)$.
\end{lemma}

The proof of the next lemma runs along the same lines as in \cite[Lemma 4.4]{BHSconnected}.

\begin{lemma}\label{rigid embedding}
    Let $X$ be an affine scheme over a field $k$. Then there exists a closed embedding of $X$ into a smooth scheme $T$ over $k$ such that for any $n$, if $H:\mathbb A^1_k\rightarrow \mathcal S^n(T)$ is an $\mathbb A^1$-homotopy containing $[x]_n$ in its image for some $x\in X(k)$, then $H$ factors through $\mathcal S^n(X)\rightarrow\mathcal S^n(T)$.
\end{lemma}

\begin{proof}
Suppose $X\subset \mathbb A^n_k$ is defined by the ideal $\langle f_1,\dots,f_r\rangle\subset k[x_1,\dots x_n]$. Consider the map $f:\mathbb A^n_k\rightarrow \mathbb A^r_k$ which is given by $(x_1,\dots,x_n)\mapsto(f_1,\dots,f_r)$. Then the fiber of $f$ at $(0,\dots,0)$ is $X$.
Now, choose some etale map $g:C\rightarrow \mathbb A^1_k$ such that $C$ is a smooth curve of positive genus and $0\in\mathbb A^1_k$ has a unique preimage say $c$. Then, $C^r\xrightarrow{g^r} \mathbb A^r_k$ is etale. Define $T:=\mathbb A^n_k\times_{f,\mathbb A^r_k,g^r}V_1^r.$ Then the fiber of the map $T\rightarrow V_1^r$ over $(c,\dots,c)$ is $X$. Clearly, $T$ is a smooth scheme containing $X$. 

We claim that $T$ is the required smooth scheme. Suppose $H$ is an $\mathbb A^1$-homotopy in $\mathcal S^n(T)$ whose image contains $[x]_n$ for some  $x\in X(k)$. Since $C^r$ is $\mathbb A^1$-invariant, by Lemma \ref{homotopies respect fibers}, there exists $Q:\text{Spec }k\rightarrow C^r$ such that the  $\mathbb A^1$-homotopy $H$ factors through $\mathcal S^n(T\times_{C^r,Q}\text{Spec }k)\rightarrow \mathcal S^n(T)$. Since $[x]_n$ is contained in the image of $H$ and $x\in X(k)$, the point $(c,\dots,c)$ belongs to the image of the composition $\mathbb A^1_k\xrightarrow{H}\mathcal S^n(T)\rightarrow C^r$. Hence, $Q=(c,\dots,c)$, and $H$ factors through the map $\mathcal S^n(X)\rightarrow\mathcal S^n(T)$.
\end{proof}

\begin{proposition}
For every $n\in N$, there exists a smooth variety $Z_n$ over $\mathbb C$, such that $\mathcal S^n(Z_n)(\mathbb C)\neq\mathcal S^{n+1}(Z_n)(\mathbb C)$.
\end{proposition}

\begin{proof}
For every $n$ and $Y_n$(see Construction \ref{construct Y_n}), let $T_n$ be the smooth variety corresponding to $Y_n$ arising from Lemma \ref{rigid embedding}. Then there exists a morphism $\gamma_n$ from $T_n$ to an $\mathbb A^1$-rigid scheme $V_n$ and a point $P_n\in V_n(\mathbb C)$ such that fiber of the morphism $\gamma_n$ over $P_n$ is $Y_n$. Choose a suitable open subscheme $Z_n$ of $T_n$ such that $\gamma_n|_{Z_n}^{-1}(P_n)=X_n$. Since any $\mathbb A^1_k\rightarrow \mathcal S^n(Z_n)$ whose image contains $[\alpha_n]_n$ factors through the map $\mathcal S^n(X_n)\rightarrow \mathcal S^n(Z_n)$, we have $[\alpha_n]_n\neq[\beta_n]_n$ in $\mathcal S^n(Z_n)(\mathbb C)$, while  $[\alpha_n]_{n+1}=[\beta_n]_{n+1}$ in $\mathcal S^{n+1}(Z_n)(\mathbb C)$.
\end{proof}

\subsection{$\mathbb A^1$-connectedness of $X_n$}

In this subsection, we show that the sequence of sheaves $(\mathcal S^m(X_{n}))_{m\geq 1}$ stabilises at the $n+2$ stage, and that 
\[
\pi_0^{\mathbb A^1}(X_n)=\mathcal S^{n+2}(X_n)=*.
\] 

\begin{theorem}
Let $k$ be any finitely generated field extension of $\mathbb C$ and let $n\geq 1$. Then  $\mathcal S^{n+1}(X_{n})(k)=*$ .   
\end{theorem}
   
\begin{proof}
   
It suffices to show that $\mathcal S^{n+1}(X_n\times_{\mathbb C}\text{Spec } k)(k)=*$. The maps $\mathcal F_n$ constructed in Theorem \ref{homotopy} will be used to prove the theorem. We will abuse the notation and write $X_n$ for the schemes $X_n\times_{\mathbb C}\text{Spec } k$, and $\mathcal F_n$ for the maps $\mathcal F_n\times_{\mathbb C}\text{Spec } k: X_n\times_{\mathbb C}\text{Spec } k\rightarrow \mathcal S(X_{n+1}\times_{\mathbb C}\text{Spec } k)$, applying same conventions to all the maps involved in construction of $\mathcal F_n$. We prove that $\mathcal S^{n+1}(X_n)(k)=*$ by induction on $n$. 

\begin{case}[$n=1$] We claim that $[Q]_2=[\beta_1]_2$ for any $Q\in X_1(k)$.
If $\mathcal F_0(k):\mathbb A^1_k(k)\rightarrow \mathcal S(X_1)(k)$ is surjective, there would be nothing to prove. However, since this is not the case, we will first list  all $k$-valued points of $X_1$, that do not admit a lift to $\mathbb A^1_k(k)$ via $\mathcal F_0(k)$. We will then slightly modify $\mathcal F_0$ to $\mathcal F_0^a$ or $\tilde{\mathcal F_0^a}$ such that their images contain $[\beta_1]_1$, and union of their images contain all $k$-valued points of $\mathcal S(X_1)$.

Recall that the map $\mathcal F_0$ includes the following morphisms:
\begin{itemize}
    \item $h_i:V_i\rightarrow X_1$, for $i=1,2,$ and
    \item $\mathcal H:\mathbb A^1_k\times W\rightarrow X_1$, where $\mathcal H$ is the restriction of $\rho_1$.
\end{itemize}
If $Q\in X_1(k)$ is contained in the image of any of these maps, then $[Q]_1$ is in the image of $\mathcal F_0$, and we have nothing to prove.

Now, the map  $\rho_1:\mathbb A^1_k\times E\rightarrow Y_1$ is an isomorphism restricted to $\mathbb A^1_k\setminus\{0\}\times E$, while the image of $\rho_1(0)\setminus\{(0,0,0)\}$ coincides with $h_2(V_2)$. Since $V_1=E\setminus\{(\lambda_1,0)^3_{i=1}, (0,-\lambda)\}$, $h_1=\rho(1)|_{V_1}$, and $W=V_1\setminus\{(0,\lambda)\}$, the remaining $k$-valued points of $X_1$ must belong to one of the following sets:
\begin{itemize}
        \item  $\rho_1(\mathbb A^1_k\setminus\{0,1\}\times \{(0,\lambda)\})$,
        \item  $\rho_1(\mathbb A^1_k\setminus\{0\}\times \{(0,- \lambda)\})$, or
        \item $\rho_1(\mathbb A^1_k\setminus\{0\}\times \{{( \lambda_i,0)}^{3}_{i=1}\})$.
\end{itemize}
Let $Q$ belong to any of the sets above. Then, the case $n = 1$ will be completed by proving $[Q]_2=[\beta_1]_2$. Let $a\in k^*$.

 If $Q=\rho_1((a,0,\lambda))$, define the map $h_1^a:V_1\rightarrow X_1$ to be $\rho_1(a)|_{V_1}$. Then, $h^a_1(a)(0,\lambda)=Q$. Now, replace $h_1$ by $h^1_a$ in the construction of $\mathcal F_0$, while keeping all the other maps the same. Since $\mathcal H(a)=h^a_1|_W$ and $\mathcal H(0)=h_2|_W$, the maps $h_1^a$ and $h_2$ can be glued together to produce the map $\mathcal F_0^a:\mathbb A^1_k\rightarrow S(X_1)$. Since $h_2(1)=\beta_1$, $\mathcal F_0^a$ contains both  $[Q]_1$ and $[\beta_1]_1$ in its image. Hence, we have $[Q]_2=[\beta_1]_2$ as desired. 

If $Q=\rho_1((a,0,-\lambda))$, define $\tilde{V_1}=E\setminus\{(\lambda_1,0)^3_{i=1}, (0,\lambda)\}$ and let $\tilde{h}_1^a=\rho(a)|_{\tilde{V_1}}$. Then $\tilde{h}^a_1(a)(0,-\lambda)=Q$. Use $\tilde{V_1}\sqcup V_2$
as the Nisnevich cover instead of $V_1\sqcup V_2$. Thus, $\tilde{V_1}\times_{\mathbb A^1_k}V_2=W$, $\mathcal H(a)=\tilde{h}^a_1|_W$ and $\mathcal H(0)=h_2|_W$.
Hence, similar to $\mathcal F_0^a$, we obtain $\tilde{\mathcal F_0^a}$ by gluing $\tilde{h}_1^a$ and $h_2$, which will contain both $[Q]_1$ and $[\beta_1]_1$ in its image.

%Since $\lambda_i\neq 0$, $\rho_1((0,\lambda_i,0))\in X_1(k)$. 
If $Q=\rho_1((a,\lambda_i,0))$, then $[Q]_1=[(0,\lambda_i,0])_1$ via the $\mathbb A^1_k$ corresponding to $\rho_1|_{\mathbb A^1_k\times\{(\lambda_i,0)\}}$. Since $(0,\lambda_i,0)\in h_2(V_2)$, we have $[(0,\lambda_i,0)]_2=[\beta_1]_2$ via $\mathcal F_0$, hence the claim is proved.
\end{case}

\begin{case}[$n>1$] Assuming the statement of theorem for $n$, we will prove it for $n+1$. It is sufficient to prove the following claim.\\

\noindent\textbf{Claim:} $[Q]_{n+2}=[\beta_{n+1}]_{n+2} $ for any $Q\in X_{n+1}(k)$.\\ 

\noindent Since $X_{n+1}=X_1\times_{\mathbb A^1_k}X_{n}$, we can write $Q$ as $(Q_1,Q_2)$, for some $Q_1\in X_1(k)$ and $Q_2\in X_n(k)$. Using notations from the case $n=1$, we see that 
\[
Q_1 \in \rho\left(\mathbb{A}^1_k \times \left\{(\lambda_i, 0)\right\}_{i=1}^3 \right) \cup \left(\bigcup_{a \in k^*} \left(h_1^a(V_1) \cup h_1^a(\tilde{V_1})\right)\right)\cup h_2(V_2).
\]
 If $[Q_1]_1\in h_1^a(V_1)$, we can modify $\mathcal F_{n}$ to $\mathcal F_{n}':X_{n} \rightarrow \mathcal S(X_{n+1}) $ such that $[Q]_1$ and $[\beta_{n+1}]_1$ are in the image of $\mathcal F_{n}'$. The map $\mathcal F_{n}$ was constructed by gluing $h_i\times id:V_i\times_{\mathbb A^1_k}X_n\rightarrow X_1\times_{\mathbb A^1_k}X_n$ for $i=1,2$. Replacing $h_1\times id$ with $h^a_1\times id$ in the construction of $\mathcal F_n$, $h^a_1\times id$ and $h_2\times id$ can be glued to give the required $\mathcal F_n'$. Then $[\beta]_1$ and $[Q]_1$ can be lifted to $X_{n}(k)$ via $\mathcal F'_n$ and since $\mathcal S^{n+1}(X_{n})(k)=*$ by the induction hypothesis, we obtain $[\beta_{n+1}]_{n+2}=[Q]_{n+2}$.
 
 For $[Q_1]_1\in \tilde h_1^a(V_1)$, replace the Nisnevich cover $V_1\sqcup V_2$ with $\tilde V_1\sqcup V_2$ and replace $h_1\times id$ with $\tilde h^a_1\times id$ in the construction of $\mathcal F_n$ and the rest of the argument proceeds similarly as for the case $[Q_1]_1\in h_1^a(V_1)$. If $[Q_1]_1\in h_2^a(V_2)$, then $[Q]_1$ is contained in the image of $\mathcal F_n$ and we are done.
 
 Finally suppose  $Q_1$ is contained in the image of $\rho_1|_{\mathbb A^1_k\times \{( \lambda_i,0)\}}$ for some $i\in\{1,2,3\}$. Then $(Q_1,Q_2)$ and $((0,\lambda_i,0),Q_2)$ are in the image of the map \[\mathbb A^1_k\xrightarrow{(\rho_1|_{\mathbb A^1_k\times \{( \lambda_i,0)\}}, Q_2)}X_1\times_{\mathbb A^1_k}X_{n}.\] Since $((0,\lambda_i,0),Q_2)$ is in the image of $h_2\times id$, we obtain \[[((0,\lambda_i,0),Q_2)]_{n+2}=[\beta_{n+1}]_{n+2},\] which further implies that $[\beta_{n+1}]_{n+2}=[Q]_{n+2}$.
\end{case}\end{proof}

\begin{theorem}\label{A1-connectedness}
    $\pi_0^{\mathbb A^1}(X_n)=\mathcal S^{n+2}(X_n)=*$.
\end{theorem}

\begin{proof}
Since $\mathcal S^{n+1}(X_n)(k)=*$ for any finitely generated and separable extension $k$ of $\mathbb C$, we have $\mathcal S^{n+2}(X_n)=*$ from Theorem \ref{S^2}.
Combining Theorem \ref{formula} with Lemma \ref{stable}, we conclude that $\pi_0^{\mathbb A^1}(X_n)=*$. 
\end{proof}

\end{document}